\documentclass[11pt]{amsart}
\usepackage[utf8x]{inputenc}
\usepackage{amsfonts}
\usepackage{amssymb}
\usepackage{amsmath}
\usepackage{delarray}
\newtheorem{thm}{Theorem}[section]

\newtheorem{lemma}[thm]{Lemma}
\newtheorem{prop}[thm]{Proposition}
\theoremstyle{definition}

\newtheorem{defn}[thm]{Definition}
\newtheorem{remark}[thm]{Remark}
\numberwithin{equation}{section}
\newtheorem{example}[thm]{Example}

\newcommand{\bbR}{\mathbb R}

\newcommand{\bbN}{\mathbb N}

\newcommand{\bbS}{\mathbb S}

\begin{document}

\title{Entropy of geometric structures}

\author{Nguyen Tien Zung}
\address{Institut de Mathématiques de Toulouse, UMR5219, Université Toulouse 3}
\email{tienzung.nguyen@math.univ-toulouse.fr}

\date{Version 2, September 2011}
\subjclass{37J35,53D20,37G05,70K45,34C14}

\keywords{entropy, geometric structure, singular foliation, Poisson structure}%

\maketitle

\begin{abstract}
We give a notion of entropy for general gemetric structures, which generalizes well-known notions of topological 
entropy of vector fields  and geometric entropy of foliations, and which can also be applied to singular objects, 
e.g.  singular foliations, singular distributions,  and Poisson structures. 
We show some basic properties for this entropy, including the \emph{additivity property},
analogous to the additivity of Clausius--Boltzmann entropy in physics. In the case of Poisson structures, 
entropy is a new invariant of dynamical nature, which is related to the transverse structure of the characteristic 
foliation by symplectic leaves. 
\end{abstract}

\section{Geometric structures}

In this paper, by a {\bf geometric structure}, we mean a quadruple 
\begin{equation}
{\mathcal G} = (M, A, \|.\|, \sharp), 
\end{equation}
where $M$ is a manifold, $A$ is a vector bundle over $M$, $\|.\|$ is a Banach norm on the fibers of $A$, and 
$\sharp: A \to TM$ is a vector bundle morphism from $A$ to the tangent bundle of $M$, called the 
{\bf anchor map}.  For simplicity, in this paper we will consider only smooth finite-dimensional objects, and
will always assume that the manifold $M$ is connected. 
The manifold $M$ is the ``geo'' part, while the norm $\|.\|$ on $A$ together with the anchor map $\sharp$
form the  ``metric'' part  in the geometric structure $\mathcal{G}.$ For simplicity, in this paper we will always
assume that the manifold $M$ is connected.

The above notion of geometric structure generalizes many well-known notions in mathematics, including: 
vector fields, Riemannian and Finsler metrics, sub-Riemannian structures, singular folations, 
symplectic structures, Poisson structures,  etc.

\begin{example} \label{example:vector}
A vector field $X$ on a manifold $M$ may be viewed as a geometric structure  ${\mathcal G} = (M, A, \|.\|, \sharp)$ where
$A = \bbR \times M$ is the trivial vector bundle of rank 1, $\|.\|$ is the absolute value on (each fiber) 
$\bbR$, and the anchor map $\sharp$ maps the unit section of $A$ to $X$.
\end{example}

\begin{example}
If $M$ is equipped with a Riemannian metric, then the associated geometric structure has $A = TM$, 
the anchor map is the identity map, and the norm on $A$ is the length of the tangent vectors given by the metric.
\end{example}

\begin{example} \label{example:foliation}
If $M$ is a Riemannian manifold with a foliation $\mathcal{F}$ or a distribution $\mathcal{D}$ on it, then $A = T\mathcal{F}$
or $A = \mathcal{D}$, the anchor map is the inclusion map, and the norm is also given by the Riemannian metric.
\end{example}

\begin{example}
If $(M,\Pi)$ is a Poisson manifold with Poisson tensor $\Pi$, together with a Riemannian metric on it, then $A = T^*M$ 
is the cotangent bundle, the anchor map $\sharp : T^*M \to TM$ is the contraction map with $\Pi$, and the norm is also
given by the metric on $M$.
\end{example}

\section{Definition of entropy}

For simplicity, let us assume for the moment that $M$ is a closed compact manifold.

For each point $x \in M$, a (piece-wise smooth) path $\gamma: [0,1] \to M$ starting at $x$ 
(i.e. $\gamma(0) = x$) will be called {\bf $A$-controllable} if there is a (piecewise-smooth)
path $\Gamma: [0,1] \to A$ such that 
\begin{equation}
\sharp (\Gamma(t)) =  {d \gamma(t) \over d t} 
\end{equation}
for almost all $t \in [0,1].$ (This notion imitates the notion of A-paths in the theory of Lie algebroids and
Lie groupoids).    
Moreover, this path is said to be of {\bf $A$-speed at most $r$}, where $r$ is some 
positive number, if the path $\Gamma: [0,1] \to A$ above can be chosen such that 
\begin{equation}
\| \Gamma(t) \| \leq r 
\end{equation}
for almost all $t \in [0,1].$ By abuse of language,  a $A$-controllable path will also be called a {\bf $A$-path}.
Denote by 
\begin{equation}
\mathcal{P}(x, r) 
\end{equation}
 the set of all $A$-paths of $A$-speed at most $r$ starting from $x$.

Fix an arbitrary Riemannian metric on $M$ and denote the corresponding distance function on $M$ by $d$
(the entropy will not depend on the choice of this metric).
For any two points $x,y \in M$, put
\begin{equation}
\delta_r(x,y) = \sup_{\gamma \in \mathcal{P}(x,r)}  \inf _{\mu \in \mathcal{P}(y,r)} \sup_{t \in [0,1]} 
d(\gamma(t), \mu(t))  .
\end{equation}

Intuitively, the above quantity measures how far can $x$ get away from $y$ by running with $A$-speed at most  
$r$ if we let $y$ try to catch $x$, also with $A$-speed at most  
$r$.  A-priori, the function $\delta_r(x,y)$ is not 
symmetric in $x$ and $y$, so to make it symmetric we put

\begin{equation}
d_r(x,y) = \delta_r (x,y) + \delta_r(y,x) .
\end{equation}

\begin{prop}
For each $r > 0$, the function $d_r: M \times M \to \bbR$ is a metric on $M$, i.e. it is positive, symmetric, 
and satisfies the triangular inequality. Moreover, If $r > s > 0$ then $d_r \geq d_s \geq 2d.$
\end{prop}

\begin{proof}
The fact that $d_r$ is symmetric is obvious. The positivity of $d_r$ comes from the inequality $d_r \geq 2d$,
which in turn comes from the inequality $\sup_{t \in [0,1]} 
d(\gamma(t), \mu(t)) \geq d(\gamma(0),\mu(0) ) = d(x,y)$ for any $\gamma \in \mathcal{P}(x,r)$ and any
$\mu \in \mathcal{P}(y,r)$.

If $r > s > 0$ then any $A$-controllable path $\gamma$ of $A$-speed at most $s$ can be turned into 
an $A$-controllable path $\hat{\nu}$ of $A$-speed at most $r$ by the following reparametrization trick:
$\hat{\nu}(t) = \nu(tr/s)$ if $0 \leq t \leq s/r$ and $\hat{\nu}(t) = \nu(1)$ if $s/r < t \leq 1.$ Conversely,
any path of speed at most $r$ can also be turned into a path of speed at most $s$ by truncating the end
and reparametrizing. With this reparametrization trick, one can easily see that $\delta_r(x,y) > \delta_s(x,y)$, 
and hence $d_r(x,y) \geq d_s(x,y).$ Indeed, fix two arbitray points $x$ and $y$, 
and assume that $\delta_s(x,y)$ is attained by a path $\gamma_0 \in \mathcal{P}(x,s),$ i.e.
\begin{equation}
\delta_r(x,y) =  \inf _{\mu \in \mathcal{P}(y,s)} \sup_{t \in [0,1]} 
d(\gamma_0(t), \mu(t))  
\end{equation}
(If the sup over all $\gamma \in \mathcal{P}(x,s)$ cannot be attained, 
then replace the above equality by a near-equality with an small error term $\epsilon$ 
and then make $\epsilon$ tend to 0 in the end; the other arguments will remain the same).
Denote by $\hat{\gamma_0} \in \mathcal{P}(x,r)$ the path obtained from $\gamma$ by reparametrizing as above. Let 
$\hat{\mu_0} \in \mathcal{P}(x,r)$ be such that
\begin{equation}
 \inf _{\mu \in \mathcal{P}(y,r)} \sup_{t \in [0,1]}  d(\hat{\gamma_0}(t), \mu(t))  = 
\sup_{t \in [0,1]}  d(\hat{\gamma_0}(t), \hat{\mu_0}(t))
\end{equation}
(If such $\hat{\mu_0}$ doesn't exist then replace the above equality by a near-equality so that it exists). Denote by 
$\mu_0 \in \mathcal{P}(y,s)$ the path obtained from $\hat{\mu_0}$ by truncating the end and reparametrizing. 
Then clearly we have
\begin{equation}
 \sup_{t \in [0,1]}  d({\gamma_0}(t), {\mu_0}(t)) \leq \sup_{t \in [0,1]}  d(\hat{\gamma_0}(t), \hat{\mu_0}(t)) =
 \inf _{\mu \in \mathcal{P}(y,r)} \sup_{t \in [0,1]}  d(\hat{\gamma_0}(t), \mu(t)),
\end{equation}
which implies that
\begin{equation}
 \delta_s(x,y) = \inf _{\mu \in \mathcal{P}(y,s)} \sup_{t \in [0,1]}  d({\gamma_0}(t), \mu(t))
\leq \inf _{\mu \in \mathcal{P}(y,r)} \sup_{t \in [0,1]}  d(\hat{\gamma_0}(t), \mu(t)) \leq \delta_r(x,y)
\end{equation}

It remains to show the triangular inequality for $d_r.$ But this inequality follows easily from the triangular inequality for $d$.
\end{proof}

For each $r > 0$ and $\epsilon > 0$, we denote by
\begin{equation}
N(d_r, \epsilon) 
\end{equation}
the maximal cardinal number of a set of points in $M$ such that for any two different points $x,y$ in this set we have
\begin{equation}
d_r(x,y) \geq \epsilon. 
\end{equation}
In other words, $N(d_r, \epsilon)$ is the maximal number of {\bf $\epsilon$-separated} in $M$ with respect to $d_r$. Since
$M$ is compact and $d_r$ is a metric, this number is finite. 

It is clear that, because $d_r$ is an increasing function with respect to $r$, $N(d_r, \epsilon) $ is also an increasing
function with respect to $r$ (and is decreasing with respect to $\epsilon$). We are interested in the ``rate of expansion''
of the metric $d_r$ when $r$ tends to infinity, via the asymptotic behavior of $N(d_r, \epsilon)$. In general,
$N(d_r, \epsilon)$ may grow exponentially with respect to $r$ (it is easy to see that it cannot grow 
faster than exponentially). So we put
\begin{equation}
 h(\mathcal{G},\epsilon) = \limsup_{r \to \infty} {\ln N(d_r, \epsilon) \over r}
\end{equation}
and
\begin{equation}
h(\mathcal{G}) = \lim_{\epsilon \to 0+} h(\mathcal{G},\epsilon).
\end{equation}

Since $N(d_r, \epsilon)$ cannot grow faster than exponentially, $h(\mathcal{G})$ is either a finite positive number or zero. 
Note that $h(\mathcal{G})$ does not depend on the choice of Riemannian metric $d$ on $M$. 
Indeed, if $d'$ is another Riemannian metric, then there is a positive constant $K >$ such that $d/K \leq d' \leq Kd$, 
which implies that $N(d_r, K\epsilon) \leq N(d'_r, \epsilon) \leq N(d_r, \epsilon/K)$ for any $r > 0$ and any $\epsilon > 0$, 
which in turn implies that  
$h_d(\mathcal{G}, K\epsilon) \leq  h_{d'}(\mathcal{G},\epsilon) \leq  h_d(\mathcal{G},\epsilon/K)$, so taking the limit 
$\epsilon \to 0$ we get $h_d(\mathcal{G}) = h_{d'}(\mathcal{G}).$

\begin{defn}
 The number $h(\mathcal{G})$ defined by the above formulas is called the {\bf entropy} of the geometric structure
${\mathcal G} = (M, A, \|.\|, \sharp).$
\end{defn}

\begin{remark}
We don't claim any originality to the above definition. It is just an adaptation of well-known definitions of entropy to 
general geometric structures. As we will see in the following sections, in the case of vector fields our 
entropy is equal to 2 times the  topological entropy, and in the case of regular foliations our 
entropy coincides with the geometric entropy first introduced by Ghys--Langevin--Walczak \cite{GLW}.
The new thing here may be the observation that the same definition works for many different geometric structures,
and in particular the singular ones. For Poisson structures, the entropy seems to be an interesting invariant which
has not been studied before. 
\end{remark}

Another a-priori non-equivalent way to define entropy is as follows: Put
\begin{equation}
\Delta_r(x,y) = \sup_{\gamma \in \mathcal{P}(x,r)}  \inf _{\mu \in \mathcal{P}(y, \infty)} \sup_{t \in [0,1]} 
d(\gamma(t), \mu(t))  
\end{equation}
(i.e. there is no restriction on the $A$-speed of $\mu$), and
\begin{equation}
D_r(x,y) = \Delta_r (x,y) + \Delta_r(y,x)
\end{equation}

It is clear that $d_r \geq D_r \geq 2d$, and $D_r$ is also symmetric. However, a-priori, $D_r$ does not necessarily
satisfy the triangular inequality, so it is not necessarily a metric, but a kind of pseudo-metric on $M$. Denote by
$N(D_r,\epsilon)$ the maximal number of $\epsilon$-separated points on $M$ with respect to $D_r$, and

\begin{equation}
H(\mathcal{G}) = \lim_{\epsilon \to 0+}  \limsup_{r \to \infty} {\ln N(D_r, \epsilon) \over r}
\end{equation}

This other notion of entropy $H(\mathcal{G})$ has some advantages (some properties are easier to prove for 
$H(\mathcal{G})$ than for $h(\mathcal{G})$). In many reasonable situations (for example, if the geometric 
structure is associated to a regular foliation) one can show that $H(\mathcal{G}) = h(\mathcal{G})$. 
However, we don't know if $H(\mathcal{G})$ is always equal to $h(\mathcal{G})$. In any case, 
we always have
\begin{equation}
H(\mathcal{G}) \leq  h(\mathcal{G}),
\end{equation}
because $D_r \leq d_r.$

The notion of entropy can be  adapted to the case of (subsets of) a non-necessarily compact manifold, and to the 
local case. If the manifold $M$ is non-compact, then $N(d_r, \epsilon)$ is infinite in general 
(unless a metric $d$ with finite diameter
is chosen). To avoid this, let
\begin{equation}
 V_1 \subset V_2 \subset \hdots \subset V_n \subset \hdots \subset M
\end{equation}
be a sequence of relatively compact subsets of $M$ such that $M = \bigcup_n V_n$, and denote by $N(d_r,\epsilon, P, V_n)$ 
the maximal number of $\epsilon$--$d_r$--separated points in $V_n.$ Then put
\begin{equation}
h(\mathcal{G}) = \lim_{n \to \infty} \lim_{\epsilon \to 0} \limsup_{r \to \infty} {\ln N(d_r,\epsilon, V_n) \over r},
\end{equation}
and call it the entropy of  $\mathcal{G}$ on $M$. (It is clear that the above definition does not depend on the choice
of the sequence $(V_n)_{n \in \bbN}$).
 
If $K$ is a compact subset in a manifold $M$, then we can define the {\bf local entropy} of $G$ at $K$ as follows. Choose two 
relatively compact open neighborhoods $U, V$ of $K$ such that $V$ contains the closure of $U$. Then define 
$\delta_r(x,y,U,V)$ as before, but with the following additional condition: the paths from $x$ must stay inside $U$, and the 
paths from $y$ must stay inside $V$. Put $d_r(x,y,U,V) = \delta_r(x,y,U,V) + \delta_r(y,x,U,V)$ for $x,y \in U$. 
Denote by $N(d_r,\epsilon,U,V)$ the maximal number of $\epsilon$-separated points in $U$ with respect to
$d_r(.,.,U,V)$, and put
\begin{equation}
h_{local}(\mathcal{G},K) = \lim_{V \to K} \lim_{U \to K} \lim_{\epsilon \to 0+} \limsup_{r \to \infty} 
{\ln N(d_r,\epsilon,U,V) \over r}.  
\end{equation}

For example,  the local entropy of a geometric structure associated to a regular foliation at a point is zero. But it is easy to
construct examples of a vector field with a singular point such that the local entropy of the corresponding geometric structure
at that singular point is non-zero.  For example, one may consider a vector field on $\bbR^n$ which commute s
with the radial vector field  $\sum_i x_i {\partial} / \partial x_i$, and which is tangent to the spheres centered at the origin,
and such that its topological entropy on a sphere is non-zero, and invoke Theorem \ref{thm:vector} of the next section.

\section{Some basic properties}

The entropy is a homogeneous function of degree -1 on the norm:

\begin{prop}[Homogeneity]
If we multiply the norm of a geometric structure $ (M,A, \|.\|, \sharp)$ by a positive constant $\gamma$,
then its entropy  will be divided by the same factor $\gamma$, i.e.
\begin{equation}
 h(M,A, \gamma \|.\|, \sharp) = h(M,A, \|.\|, \sharp) / \gamma.
\end{equation}
\end{prop}

\begin{proof}
It follows immediately from the equality 
\begin{equation}
d_r^\mathcal{G} = d_{\gamma r}^{\mathcal{G}'},
\end{equation}
 where $d_r^\mathcal{G}$ is the $d_r$ metric associated to $\mathcal{G} = (M,A, \|.\|, \sharp)$, and
$d_{\gamma r}^{\mathcal{G}'}$ is the $d_{\gamma r}$ metric associated to $\mathcal{G}' = (M,A, \gamma \|.\|, \sharp)$
(and the same inital metric $d$ on $M$).
\end{proof}

\begin{remark}
 The above proposition is also true for $H$. Moreover, it is easy to see that $H$ is monotone decreasing  with respect to
the norm, i.e. if  $\mathcal{G}_1 = (M,A, \|.\|_1, \sharp)$ and $\mathcal{G}_2 = (M,A, \|.\|_2, \sharp)$ such that
$\|.\|_1 \geq \|.\|_2$, then $D_r^{G_1} \leq D_r^{G_2}$ for any $r > 0$, therefore 
$H(\mathcal{G}_1) \leq H(\mathcal{G}_2).$ As a consequence, the property of having $H$ equal to zero doesn't
depend on the choice of the norm. We don't know if the same is true in general for the entropy $h(\mathcal{G})$ (except
in good cases, when we know that $h = H$).
\end{remark}

The following lemma is useful for detecting zero entropy:

\begin{lemma} \label{lemma:estimate}
 Let ${\mathcal G} = (M, A, \|.\|, \sharp)$ be a geometric structure on a compact manifold $M$. 
Then there is a positive constant $K$ such that, if $\rho$ is an arbitrary positive number and $x$ and $y$ are
two arbitrary points of $M$ which can be connected by a $A$-path (of time interval $[0,1]$) of speed at most $\rho$,
then for any $r > 0$ we have
\begin{equation} \label{eqn:dr_bounded}
D_r(x,y) \leq  d_r(x,y) \leq 2 K \rho.
\end{equation}
\end{lemma}

\begin{proof}
It is enough to prove that
\begin{equation}
\delta_r(x,y) \leq K \rho. 
\end{equation}
 (The same will be true for $\delta_r(y,x)$). Fix a $A$-path $\eta: [0,1] \to M$ of speed at most $\rho$ which connects $y$ to $x$,
i.e. $\eta(0) = y$ and $\eta(1) = x$. Then for any path $\gamma \in \mathcal{P}(x,r)$, assuming $r > \rho$, we can construct
a path $\mu \in \mathcal{P}(y,r)$ as follows: $\mu(t) = \eta(r t / \rho)$ if $0 \leq t \leq \rho/r$, and $\mu(t) = \gamma(t- \rho/r)$
if $\rho/r < t \leq 1.$ In other words, we let $y$ try to catch $x$ by first following the path $\eta$ to get to $x$, 
and then following the same path as $x$. By doing so, $y$ always ``lags behind'' $x$ by a time amount equal to 
$\rho/r$ in its pursuit of $x$ (and its $A$-speed is at most $r$). It implies that we can put
\begin{equation}
K = \sup_{\rho > 0}  \sup_{\zeta \in \mathcal{P}(\rho)} { d(\zeta(1), \zeta(0))  \over \rho } ,
\end{equation}
where $\mathcal{P}(\rho)$ is the set of all $A$-paths of speed at most $\rho$. The above number $K$ is finite due to the 
compactness and smoothness of our structure.
\end{proof}

The above lemma shows that, for points which can be connected by $A$-paths, the $d_r$ distance doesn't grow much at all when
$r$ tends to infinity. So, intuitively, for $d_r$ to grow exponentially, we need points which cannot be connected by $A$-paths.
In terms of control theory, we say that $M$ is {\bf controllable} by $\mathcal{G}$ if any two points on $M$ can be connected
by a $A$-path. So intuitively, the entropy must vanish in the controllable case. We have the following precise statement,
which is probably not the optimal one, but which shows the idea clearly:

Recall that a regular (smooth) distribution $\mathcal{D}$ on a manifold $M$ is nothing but a subbundle of the tangent
bundle $TM$. One says that $\mathcal{D}$ satisfies the bracket-generating condition if  repeated Lie brackets of 
vector fields tangent to $\mathcal{D}$ generate the whole $TM$ linearly. 

\begin{thm}
Let $\mathcal{D}$ be a regular distribution on $M$  which satisfies the bracket-generating condition.
Denote by $\mathcal{G}^\mathcal{D}$ the geometric structure associated to $\mathcal{D}$ 
(i.e., the vector bundle is $\mathcal{D}$, the anchor map is the inclusion map)  with an arbitrary norm. 
Then we have
\begin{equation}
h(\mathcal{G}^\mathcal{D}) = H(\mathcal{G}^\mathcal{D}) = 0. 
\end{equation}
\end{thm}

\begin{proof}
The proof follows easily from Lemma \ref{lemma:estimate} and the \emph{Ball-Box Theorem} in sub-Riemannian geometry. 
The Ball-Box Theorem (see, e.g., \cite{Gromov-BallBox}) says that, under the bracket-generating
condition, and given a Riemannian metric on the manifold, at each point of the manifold there is a local coordinate system 
$(x_1,\hdots,x_m)$ such that, for any  sufficiently small positive  $\epsilon$, 
the ball of radius $\epsilon$  with respect to the corresponding sub-Riemannnian metric 
centered at that point contains a box 
$ [- c\epsilon^{k_1}, c\epsilon^{k_1}] \times \hdots \times [- c\epsilon^{k_m}, c\epsilon^{k_m}]$
with respect to that local coordinate system, and is contained in a similar box (with another coefficient $c$), where 
$k_1,\hdots, k_m$ are positive integers and $c$ is a positive constant which do not depend on $\epsilon$. 
This Ball-Box Theorem implies in particular that, if the
distribution $\mathcal{D}$ is bracket-generating, then there is a finite number $n > 0$ such that, 
for any $\rho > 0$ small enough and any $x,y$ in M such that $d(x,y) < \rho^n$, then $x$ can be connected to $y$ by a  
$\mathcal{D}$-path of speed at most $\rho$. Together with Lemma \ref{lemma:estimate}, it implies that
$N(d_r, \epsilon)$ is bounded for each $\epsilon > 0$ (i.e. it doesn't grow to infinity at all when $r$ goes to infinity).
Therefore the entropy is zero.
\end{proof}

\begin{remark}
A special case of the above theorem, namely the case when $\mathcal{D}$ is a contact structure,
was obtained earlier by Bi\'s in \cite{Bis-Contact} (modulo the fact that, in the case of regular distributions, 
our definition of entropy is equivalent to Bi\'s' definition \cite{Bis-Contact}). 
\end{remark}

Another simple consequence of Lemma \ref{lemma:estimate} is the following:

\begin{thm}
If the anchor map $\sharp: A \to TM$ of a geometric structure  ${\mathcal G} = (M, A, \|.\|, \sharp)$ is surjective, 
then the entropy of $\mathcal{G}$ is zero.
\end{thm}

In particular, the entropy of a symplectic structure is zero. The proof follows immediately from Inequaltiy 
(\ref{eqn:dr_bounded}), which shows that the metrics $d_r$ are uniformly bounded with respect to $r$ in this case,
so no exponential growth (not even polynomial growth).

Intuitively, the entropy is mainly an invariant of the \emph{transverse structure} of  the ``singular foliation'' generated 
by the geometric structure, where each leaf is defined as the set of points which can be reached from a given point
by  $A$-paths. If this transversal structure is ``tame'' then the entropy will probably be zero, otherwise it will probably be
positive. 

Examples of geometric structure with non-zero entropy are provided by vector fields with non-zero 
topological entropy (it is well known in dynamical systems that most vector fields, 
especially the ones with hyperbolic behavior, have non-zero entropy), and the following theorem:

\begin{thm}[Topological entropy] \label{thm:vector}
 Let $X$ be a smooth vector field on a compact manifold $M$, and denote by $\mathcal{G}_X$ the associated
geometric structure (see Example \ref{example:vector}). Then we have
\begin{equation}
 h(\mathcal{G}_X) = H(\mathcal{G}_X) = 2 h_{top} (X),
\end{equation}
where $h_{top}(X)$ denotes the topological entropy of $X$.
\end{thm}

\begin{proof}
 Recall the following formula, due to  Bowen and Dinaburg (see, e.g., \cite{Dinaburg-Entropy}), for topological
entropy: Denote by $N(X,r,\epsilon)$ the maximal number of $\epsilon$-separated points with respect to $d^X_r$ on $M$,
where 
\begin{equation}
 d^X_r (x,y) = \max_{t \in [0,r]} d(\phi^t_X(x), \phi^t(y)),
\end{equation}
where $d$ is a fixed Riemannian metric on $M$, and $\phi^t_X$ is the flow of the vector field $X$. Then
\begin{equation}
 h_{top} (X) = \lim_{\epsilon \to 0+} \limsup_{r \to \infty}  {\ln N(X,r,\epsilon) \over r} .
\end{equation}
A $A$-path with respect to $\mathcal{G}_X$ is simply a piece of an orbit of $X$ on which a point $x$ can move 
back and forth. If a point $x$ moves back and forth, then a point $y$ which pursuits it can also imitate the same back 
and forth movement (i.e. if $\gamma$ is a path from $x$ then we can also
choose a path $\mu$ from $y$ in pursuit of $x$ such that $\mu(t) = \mu(s)$ whenever $\gamma(t) = \gamma(s)$).
So the ``maximal escape'' is achieved when $x$ either moves forwards all the time or moves backwards
all the time, with maximal possible speed. In other words, we have
\begin{equation}
\delta_r(x,y) = \max \left(  \inf _{\mu \in \mathcal{P}(y,r)} \sup_{t \in [0,1]} 
d(\phi^{tr}_X(x), \mu(t)) , \inf _{\mu \in \mathcal{P}(y,r)} \sup_{t \in [0,1]} 
d(\phi^{-tr}_X(x), \mu(t)) \right),
\end{equation}  
and a similar formula for $\Delta_r(x,y)$.  It is clear that
\begin{equation}
\sup_{t \in [0,1]} 
d(\phi^{tr}_X(x), \phi^{tr}_X(y)) \geq \inf _{\mu \in \mathcal{P}(y,r)} \sup_{t \in [0,1]} 
d(\phi^{tr}_X(x), \mu(t)) \geq \inf _{\mu \in \mathcal{P}(y)} \sup_{t \in [0,1]} 
d(\phi^{tr}_X(x), \mu(t)). 
\end{equation}
On the other hand, we have the following simple lemma, whose proof is a direct consequence of the (uniform version of)
the rectification theorem for vector fields:
\begin{lemma}
 There are positive constants $K,k$ such that, if 
$$\inf _{\mu \in \mathcal{P}(y)} \sup_{t \in [0,1]} 
d(\phi^{tr}_X(x), \mu(t)) = \epsilon < k$$ 
and $$d(x,y) < {\epsilon \over Kr + 1},$$ 
then $$d(\phi^{tr}_X(x), \phi^{tr}_X(y)) \leq 2 \epsilon .$$
\end{lemma}
In other words, the ``nearly optimal pursuit path'' is already provided by $\phi^{tr}_X(y).$
It follows from the above inequalities that $d_{r}(x,y)$ and $D_r(x,y)$ are ``comparable'' to
\begin{equation}
\mathbb{D}_r(x,y) := \max \left( \sup_{t \in [0,1]} 
d(\phi^{tr}_X(x), \phi^{tr}_X(y)), \sup_{t \in [0,1]} 
d(\phi^{-tr}_X(x), \phi^{-tr}_X(y)) \right).  
\end{equation}
More precisely, we have
\begin{equation}
2 \mathbb{D}_r(x,y) \geq d_r(x,y) \geq D_r(x,y),
\end{equation}
and there exist constants $K, k > 0$ (which do not depend on $r$) such that if $D_r(x,y) \leq k$ then
\begin{equation}
\mathbb{D}_r(x,y) \leq 2 \max \left(D_r(x,y), (Kr +1) d(x,y) \right)
 \end{equation}
Since we are interested only in the exponential behavior of the (pseudo)metrics with respect to $r$, the terms
$(Kr+1)d(x,y)$ in the last inequality can be ignored, because it is only polynomial in $r$. It implies that
$d_r, D_r$ and $\mathbb{D}_r$ give the same entropy. Now $\mathbb{D}_r$ give twice the topological
entropy, because by letting the points move both forwards and backwards, we have ``doubled the time'':
\begin{equation}
\mathbb{D}_r(x,y) = \sup_{t \in [0,2]} 
d(\phi^{tr}_X( \phi^{-r}_X(x)), \phi^{tr}_X( \phi^{-r}_X(y)))
\end{equation}
\end{proof}

More generally, we have:

\begin{thm}[Geometric entropy] \label{thm:foliation}
Let $\mathcal{F}$ be a regular foliation on a Riemannian manifold $(M,g)$, and denote by
$\mathcal{G}_\mathcal{F}$ the associated geometric structure of $\mathcal{F}$ (together with the norm
coming from the metric). Then
\begin{equation}
 h(\mathcal{G}_\mathcal{F}) = H(\mathcal{G}_\mathcal{F}) = h_{GLW}(\mathcal{F}),
\end{equation}
where $h_{GLW}$ denotes the geometric entropy introduced 
by Ghys--Langevin--Walczak \cite{GLW} for regular foliations.
\end{thm}

See \cite{GLW,Walczak-Foliations} for formulas and properties of geometric entropy of foliations. The proof of 
Theorem \ref{thm:foliation} is absolutely similar to the proof of Theorem \ref{thm:vector}: one can show that
the ``nearly optimal puirsuit'' is obtained by orthogonal projection of a path (starting from $x$) on a leaf
(passing through $y$), and that orthogonal projection is used in one of the equivalent definitions 
of geometric entropy in \cite{GLW}.

Another basic property of our entropy is the following additivity.
Let ${\mathcal G_1} = (M_1, A_1, \|.\|_1, \sharp_1)$ and ${\mathcal G_2} = (M_2, A_2, \|.\|_2, \sharp_2)$ 
be two arbitrary  geometric structures. Denote by ${\mathcal G_1}\oplus{\mathcal G_2}$ their 
{\bf direct sum} $(M_1 \times M_2, A_1 \oplus A_2,  \max( \|.\|_1, \|.\|_2), \sharp_1 \oplus \sharp_2)$ 
(here we use the \emph{max norm}, i.e. the maximal of the norms of the two components 
as the norm on the direct sum).

\begin{thm}[Additivity] We have
\begin{equation}
h(\mathcal{G}_1 \oplus \mathcal{G}_2) =  h(\mathcal{G}_1) + h(\mathcal{G}_2).
\end{equation}
\end{thm}

\begin{proof}
 Fix metrics $d_1$ and $d_2$ on $M_1$ and $M_2$, and use the max metric $\max(d_1,d_2)$ on the product 
$M_1 \times M_2$ (this last metric is not Riemiannian, but it doesn't matter). Then
$$ \delta_r^{\mathcal{G}_1 \oplus \mathcal{G}_2} = \max (\delta_r^{\mathcal{G}_1}, \delta_r^{\mathcal{G}_2}),$$
because the pursuits in the two components are independent, and the escape in the product manifold is the maximal
of the two escapes in the two components. Therefore
$$ d_r^{\mathcal{G}_1 \oplus \mathcal{G}_2} = \max (d_r^{\mathcal{G}_1}, d_r^{\mathcal{G}_2}).$$
It follows that $N^{\mathcal{G}_1 \oplus \mathcal{G}_2}(d_r,\epsilon)$ is comparable to
$N^{\mathcal{G}_1}(d_r,\epsilon) \times N^{\mathcal{G}_2}(d_r,\epsilon).$ Indeed, if there is a set $S_1$ of
$\epsilon$--$d_r^{\mathcal{G}_1}$--separated points in $M_1$ and a set $S_2$ of
$\epsilon$--$d_r^{\mathcal{G}_2}$--separated points in $M_2$, then their direct product will be a set of
$\epsilon$--$d_r^{\mathcal{G}_1 \oplus \mathcal{G}_2}$--separated points in $M_1 \times M_2$, therefore 
$$N^{\mathcal{G}_1 \oplus \mathcal{G}_2}(d_r,\epsilon) \geq 
N^{\mathcal{G}_1}(d_r,\epsilon) \times N^{\mathcal{G}_2}(d_r,\epsilon).$$
Conversely, for each 
domain $U$ of $M_1$ of $d_r$--diameter smaller than $\epsilon$, there are at most $N^{\mathcal{G}_2}(d_r,\epsilon)$
$\epsilon$--$d_r^{\mathcal{G}_1 \oplus \mathcal{G}_2}$--separated points in $U \times M_2$. The minimal number of 
domains of $d_r$--diameter smaller than $\epsilon$ needed to cover $M_1$ is bounded above by
$N^{\mathcal{G}_1}(d_r,\epsilon/2)$, therefore
$$N^{\mathcal{G}_1 \oplus \mathcal{G}_2}(d_r,\epsilon) \leq 
N^{\mathcal{G}_1}(d_r,\epsilon/2) \times N^{\mathcal{G}_2}(d_r,\epsilon).$$
The theorem follows directly from these two inequalities.
\end{proof}

\begin{remark}
The above additivity property is analogous to the well-known additivity of Clausius--Boltzmann entropy in physics and 
the Shannon entropy in information theory. If instead of the max norm, we use another norm on the direct sum
(for example the sum norm), then instead of the additivity we may get some kind of sub-additivity for the entropy.
\end{remark}

\begin{remark}
One can probably have an upper bound for $h(\mathcal{G})$ by some kind of  ``maximal Lyapunov exponent'' or 
``sum of positive Lyapounov exponents'' of the geometric structure, similarly to the case of 
dynamical systems \cite{Pesin_LyapunovExponent}, but we will not enter that direction in this paper. 
\end{remark}

\section{Entropy of Poisson structures}

To define the entropy of a Poisson structure, we need an additional ingredient, namely a norm on the cotangent bundle
of the manifold (or a Riemannian metric on the manifold). As we mentioned earlier, the entropy mainly
reflects the transverse structure of the leaves created by $A$-paths, which are nothing but the symplectic leaves
of our Poisson structure in this case. So if the associated characteristic foliation by symplectic leaves 
(see, e.g., Chapter 1 of \cite{DufourZung-Poisson}) is ``complicated,
twisted'' then we would expect positive entropy, while if this foliation is ``simple'' then we would expect zero entropy.

It is easy to construct Poisson structure with non-zero entropy. For example, let $X$ be a
vector field on a manifold $N$. Put $M = N \times \bbS^1$, and 
\begin{equation}
\Pi_X = X \wedge \partial / \partial q 
\end{equation}
where $q$ is the periodic coordinate (modulo 1) on $\bbS^1.$ Then $\Pi_X$ is a Poisson structure of rank 2 on $M$.
Let $\|.\|_1$ be an arbitrary norm on $T^*N$, $\|.\|_2$ be the standard norm on $T^* \bbS^1$ such that $\|dq\|_2 = 1$, and
denote by $\|.\|$ the max norm on $T^*M$ generated by $\|.\|_1$ and $\|.\|_2$. Denote by $\mathcal{G}$ the geometric 
structure associated to $\Pi_X$ and this max norm on $M$. 

\begin{thm}
We have $h(\mathcal{G}) = H(\mathcal{G}) = 2 h_{top}(X)$. 
In particular,  $h(\mathcal{G}) \neq 0$ if and only if $h_{top}(X) \neq 0.$  
\end{thm}

The proof is straightforward and absolutely similar to the proof of Theorem \ref{thm:vector}.

There is a general idea, according to which, a dynamical system which is integrable in some natural dynamical sense
(e.g. Liouville integrability for Hamiltonian systems) must have zero entropy (outside a pathological singular invariant 
set where the system is ``not very integrable''). 
For  Poisson structures there is also a notion of integrability. However, integrability of Poisson structure (in the sense
that it can be integrated into a symplectic groupoid, see e.g. \cite{CF-PoissonIntegrable}) is not a dynamical property
but rather a geometric property. So it is not surprising that there are Poisson structures which are integrable but which
admit non-zero entropy, and on the other hand there are Poisson structures which are not integrable but which have
zero entropy.

\begin{example}
 The rank-2 Poisson structure $\Pi_X = X \wedge \partial / \partial q$ above is integrable, but will have non-zero
entropy if the vector field $X$ has non-zero entropy.
\end{example}

\begin{example}
 The Poisson structure 
\begin{equation}
 (x_1^2 + x_2^2 + x_3^2)(x_1 \partial x_2 \wedge \partial x_3 +
x_2 \partial x_3 \wedge \partial x_1 + x_3 \partial x_1 \wedge \partial x_2)
\end{equation}
 on $\bbR^3$ is non-integrable, but its symplectic
foliation is very simple (spheres centered at the origin), and it has zero entropy (with respect to any norm).
\end{example}

The non-vanishing of the entropy for Poisson structures is also a non-linear phenomenon: \emph{any linear 
Poisson structure has zero entropy with respect to any norm}. We will leave this last
statement as a conjecture. (Indication: use the fact that there is a complete set of rational
invariant functions for the singular foliation into coadjoint orbits on the dual of a Lie algebra. The proof
of this conjecture is a simple excercise for Lie algebras of compact type, and we believe that the conjecture is
also true for any other finite-dimensional Lie algebra).

\vspace{0.5cm}

{\bf Acknowledgements}. This work was first presented as part of an invited talk at the conference 
``Poisson 2010: Poisson geometry in mathematics and physics'' held at IMPA, Rio de Janeiro, July/2010. 
We would like to thank the organizers of the conference, especially Rui Fernandes and Henrique Bursztyn, 
for the invitation and financial and moral support. We would also like to thank  Alain Connes 
(for telling us the history of Boltzmann and how people mocked at his entropy formula, and the story 
about entropy of food),  Yvette Kosmann-Schwarzbach (for letting us know about Clausius),  
Frank Michael Forger (for his question about the additivity property, which leads to the additivity theorem 
in this paper),  Alexey Bolsinov, Jean Pierre Marco, Jean-Pierre Ramis, and many other colleagues for interesting
disscusions about entropy. I'm also thankful to the referees of this paper for very useful critical remarks.

\end{document}